\documentclass[a4paper]{amsart}
\usepackage[T1]{fontenc}
\usepackage{mathptmx}
\usepackage{amsmath,amssymb,amsthm,amsbsy}
\usepackage{mathrsfs}
\usepackage{graphics,graphicx,color}
\usepackage{hyperref}
\usepackage{enumerate}
\usepackage{xcolor}
\usepackage{cleveref}
\usepackage{enumitem}

\usepackage{darkmode}
\disabledarkmode

\usepackage[abbrev]{amsrefs}

\usepackage{orcidlink}

\usepackage{setspace}

\theoremstyle{plain}
\newtheorem{lemma}{Lemma}[section]
\newtheorem{theorem}[lemma]{Theorem}

\theoremstyle{definition}

\newtheorem{example}[lemma]{Example}

\theoremstyle{remark}
\newtheorem{remark}[lemma]{Remark}

\DeclareMathOperator{\im}{im}

\DeclareMathOperator{\id}{id}

\DeclareMathOperator{\reg}{Reg}
\DeclareMathOperator{\Sym}{Sym}

\begin{document}

%%%%%%%%%%%%%%%%%%%%%%%%%%%%%%title%%%%%%%%%%%%%%%%%%%%%%%%%%%%%%%%%%%%%%%%%%%%%%%%%%%%%%%%%%%%%%%%%%%%%%%%%%%%%%%%%%%%%%%%%%%%%

\title[Maximal regular subsemigroups]{Maximal regular subsemigroups containing the group of units of transformation semigroups that preserve a uniform equivalence relation with two classes}

\author[K. Sangkhanan]{Kritsada Sangkhanan\orcidlink{0000-0002-1909-7514}}
\address{Department of Mathematics, Faculty of Science, Chiang~Mai~University, Chiang~Mai, 50200, Thailand}
\thanks{Corresponding author: Kritsada Sangkhanan (kritsada.s@cmu.ac.th). This research was supported by Chiang Mai University.}
\email{kritsada.s@cmu.ac.th}

%%%%%%%%%%%%%%%%%%%%%%%%%%%abstract%%%%%%%%%%%%%%%%%%%%%%%%%%%%%%%%%%%%%%%%%%%%%%%%%%%%%%%%%%%%%%%%%%%%%%%%%%%%%%%%%%%%%%%%%%%%%

\begin{abstract}
	Let $X$ be a nonempty set and $E$ an equivalence relation on $X$. The semigroup $T_E(X)$ consists of all transformations on $X$ that preserve $E$. Although the collection of all regular elements of $T_E(X)$ is not, in general, a subsemigroup, we investigate maximal regular subsemigroups of $T_E(X)$ containing the group of units $G_E(X)$. We construct two such maximal regular subsemigroups, denoted by $\mathcal{M}_1$ and $\mathcal{M}_2$, and show that their maximality holds in full generality for an arbitrary set $X$ and an arbitrary equivalence relation $E$ that is neither the identity relation nor the universal relation on $X$ (whether finite or infinite, uniform or non-uniform). Furthermore, when $X$ is finite and $E$ is a uniform equivalence relation with exactly two classes of cardinality at least $2$, we provide a complete classification by proving that they are the only maximal regular subsemigroups containing $G_E(X)$.
\end{abstract}

%%%%%%%%%%%%%%%%%%%%%%%%%%%%title page%%%%%%%%%%%%%%%%%%%%%%%%%%%%%%%%%%%%%%%%%%%%%%%%%%%%%%%%%%%%%%%%%%%%%%%%%%%%%%%%%%%%%%%%%%%%%

\maketitle

%%%%%%%%%%%%%%%%%%%%%%%%%%%Introduction%%%%%%%%%%%%%%%%%%%%%%%%%%%%%%%%%%%%%%%%%%%%%%%%%%%%%%%%%%%%%%%%%%%%%%%%%%%%%%%%%%%%%%%%%%%%%

\section{Introduction}\label{sec:intro}

The study of algebraic semigroups of transformations has been a classical and active topic in semigroup theory for several decades (see, for instance, the foundational monographs by Clifford and Preston \cites{Clifford1961, Clifford1967} and Howie \cite{Howie1995}). Let $X$ be a nonempty set, and let $T(X)$ denote the full transformation semigroup on $X$ under composition of mappings. A fruitful research program in this area is the investigation of subsemigroups of $T(X)$ consisting of transformations that preserve certain structural relations or algebraic properties on $X$.

Among such subsemigroups, transformation semigroups preserving an equivalence relation have attracted considerable attention. In 1994, Pei \cite{Pei1994} introduced the semigroup
\[
	T_E(X) = \{\alpha \in T(X) : \forall x,y \in X,\, (x,y) \in E \implies (x\alpha, y\alpha) \in E\}
\]
of all transformations on $X$ preserving an equivalence relation $E$, showing that $T_E(X)$ is an $\alpha$-semigroup and the greatest transformation semigroup to which $E$ is suitable. Subsequently, Pei \cite{Pei2005} characterized Green's relations $\mathcal{L}$, $\mathcal{R}$, $\mathcal{H}$, and $\mathcal{D}$ on $T_E(X)$, provided sufficient conditions for $\mathcal{J}$-equivalence and a characterization in the finite-image case, and described the regular elements of $T_E(X)$, proving that $T_E(X)$ is regular if and only if $E$ is the identity or the universal relation on $X$. Later, Pei and Zhou \cite{Pei2011} investigated abundant semigroups of transformations preserving an equivalence relation, characterizing the equivalence relations $E$ for which $T_E(X)$ is abundant.

Various algebraic, order-theoretic, and topological aspects of $T_E(X)$ have since been extensively explored. Sun, Pei, and Cheng \cite{Sun2008} characterized the natural partial order on $T_E(X)$ and obtained sufficient conditions for left and right compatibility. They also characterized the maximal and minimal elements and proved the existence of upper and lower covers for non-maximal and non-minimal elements, respectively. This direction was further developed by Sun \cite{Sun2019}, who provided necessary and sufficient conditions for transformations in $T_E(X)$ to be left or right compatible with respect to the natural order. In another direction, Kaewnoi, Petapirak, and Chinram \cite{Kaewnoi2019} characterized left and right magnifying elements in $T_E(X)$.

In parallel with $T_E(X)$, considerable interest has focused on the subsemigroup of transformations preserving an equivalence relation in both directions, known as the double direction equivalence-preserving transformation semigroup:
\[
	T_{E^*}(X) = \{\alpha \in T(X) : \forall x,y \in X,\, (x,y) \in E \iff (x\alpha, y\alpha) \in E\}.
\]
The semigroup $T_{E^*}(X)$ was introduced and studied by Deng, Zeng, and Xu \cite{Deng2010}, who described its regular elements and Green's relations, establishing that $T_{E^*}(X)$ is a regular semigroup if and only if $|X/E|$ is finite. Sun and Sun \cite{Sun2013} studied the natural partial order on $T_{E^*}(X)$, determined the compatibility of elements, characterized maximal and minimal elements and investigated the existence of greatest lower bounds. Deng \cite{Deng2016} investigated Green's $*$-relations, certain $*$-ideals, and Rees quotient semigroups for $T_{E^*}(X)$, while Sawatraksa, Kammoo, and Namnak \cite{Sawatraksa2018} investigated $E$-inversive elements and described when $T_{E^*}(X)$ is an $E$-inversive semigroup. Chen, Liu, and Wang \cite{Chen2023} examined generalized Green's relations, abundant elements, and magnifying elements on $T_{E^*}(X)$, proved that $\reg(T_{E^*}(X))$ forms a subsemigroup of $T_{E^*}(X)$ (\cite{Chen2023}*{Lemma~2.5}), and also noted that it is a regular subsemigroup. Sangkhanan \cite{Sangkhanan2026}*{Theorem~2.2} explicitly established that this regular part is the largest regular subsemigroup of $T_{E^*}(X)$ and studied its Green's relations, ideals, and kernel.

Although the regular elements of $T_E(X)$ were completely characterized by Pei \cite{Pei2005}, the collection $\reg(T_E(X))$ fails to form a subsemigroup in general (as demonstrated in Example~\ref{ex: not closed} below). This poses a fundamental and challenging problem: \emph{What are the maximal regular subsemigroups of $T_E(X)$?} The study of maximal regular subsemigroups of transformation semigroups has long been an active and fruitful line of research. In particular, for the finite full transformation semigroup $T_n$ on an $n$-element set, You \cite{You2002} completely determined all maximal regular subsemigroups of $T_n$ and showed that they fall into two distinct types: those of the form $K(n, n-1) \cup G$, where $G$ is a maximal subgroup of the symmetric group $S_n$, and the unique subsemigroup $K(n, n-2) \cup S_n$ containing the entire symmetric group $S_n$, where $K(n, r) = \{\alpha \in T_n : |\im\alpha| \leq r\}$ denotes the ideal of transformations of rank at most $r$. This fundamental result demonstrates that determining maximal regular subsemigroups containing the full group of units provides a particularly natural classification program. For the semigroup $T_E(X)$, its group of units $G_E(X)$ consists of those permutations on $X$ that map each equivalence class bijectively onto an equivalence class of the same cardinality (as characterized in Theorem~\ref{thm: group of units}). In contrast to the classical setting of $T_n$, where only a single maximal regular subsemigroup containing $S_n$ exists, classifying the maximal regular subsemigroups of $T_E(X)$ containing $G_E(X)$ reveals rich new structural phenomena.

The main objective of the present paper is to address this problem. We construct two distinct maximal regular subsemigroups of $T_E(X)$ containing the group of units $G_E(X)$, denoted by $\mathcal{M}_1$ and $\mathcal{M}_2$, and establish their maximality in full generality—for any nonempty set $X$ (whether finite or infinite) and any equivalence relation $E$ on $X$ that is neither the identity relation nor the universal relation on $X$ (whether uniform or non-uniform, with an arbitrary number of classes). When $X$ is finite and the quotient set $X/E$ consists of exactly two equivalence classes of the same cardinality (that is, $E$ is a uniform equivalence relation with two classes $|A_1| = |A_2| = k \geq 2$), we establish a complete classification theorem showing that these two subsemigroups are the \emph{only} maximal regular subsemigroups of $T_E(X)$ containing $G_E(X)$. Here, the assumption $k \geq 2$ excludes the trivial case where both classes are singletons, in which case $E$ is the identity relation and $T_E(X) = T_2$ is already regular.

The paper is organized as follows. In Section~\ref{sec:prelim}, we collect necessary preliminaries, establish notation, recall the notions of regular and maximal regular subsemigroups, and characterize the group of units $G_E(X)$. In Section~\ref{sec:main}, we construct the two maximal regular subsemigroups, establish their regularity and maximality in $T_E(X)$, and present the full classification theorem for uniform equivalence relations with two equivalence classes.

%%%%%%%%%%%%%%%%%%%%%%%%%%%Preliminaries%%%%%%%%%%%%%%%%%%%%%%%%%%%%%%%%%%%%%%%%%%%%%%%%%%%%%%%%%%%%%%%%%%%%%%%%%%%%%%%%%%%%%%%%%%%%%

\section{Preliminaries}\label{sec:prelim}

We first establish some notation and recall preliminary results that will be used throughout the paper. For all undefined notions, the reader is referred to \cites{Clifford1961, Clifford1967, Howie1995}. For a subset $Y$ of a semigroup, we denote by $\langle Y \rangle$ the subsemigroup generated by $Y$. If $Y$ is a subset and $a_1, \dots, a_m$ are elements of the semigroup, we write $\langle Y, a_1, \dots, a_m \rangle$ as shorthand for $\langle Y \cup \{a_1, \dots, a_m\} \rangle$. Let $X$ be a nonempty set and $E$ an equivalence relation on $X$. The family of all equivalence classes (or \emph{$E$-classes}) of $X$ modulo $E$ forms a partition of $X$, and the quotient set is denoted by $X/E$. For each $x\in X$, the $E$-class containing $x$ is denoted by $[x]_E = \{y\in X : (x, y)\in E\}$. An equivalence relation $E$ on $X$ (or the partition $X/E$) is called \emph{uniform} if all of its equivalence classes have the same cardinality, that is, $|A| = |B|$ for all $A, B \in X/E$. For each $\alpha\in T_E(X)$ and $A\subseteq X$, the restriction of $\alpha$ to $A$ is denoted by $\alpha|_A$. We adopt the notation introduced in Clifford and Preston \cite{Clifford1967}; namely, if $\alpha\in T(X)$, then we write
\[
	\alpha = \left(
	\begin{array}{c}
		X_i \\
		a_i \\
	\end{array}
	\right)
\]
and take as understood that the subscript $i$ belongs to some unmentioned index set $I$, the abbreviation $\{a_i\}$ denotes $\{a_i : i\in I\}$, $X\alpha=\{a_i\}$, and $a_i\alpha^{-1}=X_i$. For each $\alpha \in T_E(X)$ and $A \in X/E$, the preimage $A\alpha^{-1}$ is either empty or a union of $E$-classes. Following Pei \cite{Pei2005}, we denote by
\[
	E(\alpha) = \{A\alpha^{-1} : A \in X/E,\, A\alpha^{-1} \neq \emptyset\}
\]
the partition of $X$ into the non-empty preimages of $E$-classes under $\alpha$.

Let $\Sym(X)$ denote the symmetric group on $X$, consisting of all permutations of $X$. The transformation semigroup preserving double direction equivalence is defined by
\[
	T_{E^*}(X) = \{\alpha \in T(X) : \forall x,y \in X,\, (x,y) \in E \iff (x\alpha, y\alpha) \in E\}.
\]
The semigroup $T_E(X)$ is a monoid with identity element $\id_X$. Its group of units, denoted by $G_E(X)$, is the set of all invertible elements in $T_E(X)$. We provide the following characterization of $G_E(X)$.

\begin{theorem}\label{thm: group of units}
	Let $E$ be an equivalence relation on $X$. The group of units of $T_E(X)$ is given by
	\[
		G_E(X) = T_{E^*}(X) \cap \Sym(X) = \{\alpha \in \Sym(X) : \forall A \in X/E, \exists B \in X/E \text{ such that } A\alpha = B\}.
	\]
\end{theorem}
\begin{proof}
	By definition, the group of units of $T_E(X)$ consists of all elements in $T_E(X)$ that are invertible in $T_E(X)$, that is,
	\[
		G_E(X) = \{\alpha \in T_E(X) : \exists \beta \in T_E(X), \alpha\beta = \beta\alpha = \id_X\}.
	\]
	First, let $\alpha \in G_E(X)$. Then $\alpha \in T_E(X)$ and there exists $\beta \in T_E(X)$ such that $\alpha\beta = \beta\alpha = \id_X$. This implies that $\alpha$ is a bijection on $X$, so $\alpha \in \Sym(X)$ and $\beta = \alpha^{-1} \in T_E(X)$. For each $A \in X/E$, since $\alpha \in T_E(X)$, there exists $B \in X/E$ such that $A\alpha \subseteq B$, which gives $A \subseteq B\alpha^{-1}$. Moreover, since $\alpha^{-1} \in T_E(X)$, there exists $A' \in X/E$ such that $B\alpha^{-1} \subseteq A'$. Thus $A \subseteq B\alpha^{-1} \subseteq A'$. Since $A, A' \in X/E$ and $A \cap A' \neq \emptyset$, we have $A = A'$, which implies $B\alpha^{-1} = A$, and applying $\alpha$ yields $A\alpha = B$. Therefore,
	\[
		G_E(X) \subseteq \{\alpha \in \Sym(X) : \forall A \in X/E, \exists B \in X/E \text{ such that } A\alpha = B\}.
	\]

	Next, let $\alpha \in \{\alpha \in \Sym(X) : \forall A \in X/E, \exists B \in X/E \text{ such that } A\alpha = B\}$. Let $x,y \in X$. If $(x,y) \in E$, then $x,y \in A$ for some $A \in X/E$, so $x\alpha, y\alpha \in A\alpha = B \in X/E$, which gives $(x\alpha, y\alpha) \in E$. Conversely, if $(x\alpha, y\alpha) \in E$, then $x\alpha, y\alpha \in B$ for some $B \in X/E$. Since $x \in A$ for some $A \in X/E$, by assumption we have $A\alpha = B'$ for some $B' \in X/E$. Then $x\alpha \in B \cap B'$, which implies $B = B' = A\alpha$. Since $y\alpha \in B = A\alpha$ and $\alpha$ is a bijection, we get $y \in A$, and so $x,y \in A$, which implies $(x,y) \in E$. Thus $(x,y) \in E \iff (x\alpha, y\alpha) \in E$, which proves that $\alpha \in T_{E^*}(X)$. Since $\alpha \in \Sym(X)$, we obtain $\alpha \in T_{E^*}(X) \cap \Sym(X)$.

	Finally, let $\alpha \in T_{E^*}(X) \cap \Sym(X)$. Then $\alpha \in T_E(X)$ and $\alpha^{-1} \in \Sym(X)$. For any $u,v \in X$, if $(u,v) \in E$, then since $\alpha \in T_{E^*}(X)$ and $((u\alpha^{-1})\alpha, (v\alpha^{-1})\alpha) = (u,v) \in E$, it follows that $(u\alpha^{-1}, v\alpha^{-1}) \in E$. Thus $\alpha^{-1} \in T_E(X)$, which shows that $\alpha\alpha^{-1} = \alpha^{-1}\alpha = \id_X$ with $\alpha, \alpha^{-1} \in T_E(X)$, and hence $\alpha \in G_E(X)$.

	This completes the proof.
\end{proof}

Next, we recall the concepts of regularity in semigroup theory (see \cites{Clifford1961, Howie1995}). Let $S$ be a semigroup. An element $a \in S$ is said to be \emph{regular} if there exists $x \in S$ such that $axa = a$. The set of all regular elements of $S$ is denoted by $\reg(S)$. The semigroup $S$ is called a \emph{regular semigroup} if every element of $S$ is regular, that is, $\reg(S) = S$. A subsemigroup $T$ of $S$ is called a \emph{regular subsemigroup} of $S$ if $T$ is itself a regular semigroup, that is, for each $a \in T$, there exists $x \in T$ such that $axa = a$.

A proper regular subsemigroup $M$ of $S$ is called a \emph{maximal regular subsemigroup} of $S$ if there is no regular subsemigroup $R$ of $S$ such that $M \subsetneq R \subsetneq S$. In other words, if $R$ is a regular subsemigroup of $S$ such that $M \subseteq R \subseteq S$, then $R = M$ or $R = S$. Furthermore, if $H$ is a subsemigroup of $S$, a proper regular subsemigroup $M$ of $S$ containing $H$ is called a \emph{maximal regular subsemigroup of $S$ containing $H$} if there is no regular subsemigroup $R$ of $S$ containing $H$ such that $M \subsetneq R \subsetneq S$. In other words, if $R$ is a regular subsemigroup of $S$ containing $H$ such that $M \subseteq R \subseteq S$, then $R = M$ or $R = S$.

%%%%%%%%%%%%%%%%%%%%%%%%%%%Main Results%%%%%%%%%%%%%%%%%%%%%%%%%%%%%%%%%%%%%%%%%%%%%%%%%%%%%%%%%%%%%%%%%%%%%%%%%%%%%%%%%%%%%%%%%%%%%

\section{Maximal regular subsemigroups containing the group of units}\label{sec:main}

In this section, we determine the maximal regular subsemigroups of $T_E(X)$ containing the group of units $G_E(X)$. We first recall a characterization of regular elements in $T_E(X)$ due to Pei \cite{Pei2005}.

\begin{theorem}[\cite{Pei2005}, Corollary 2.3]\label{thm: reg}
	Let $\alpha\in T_E(X)$. Then $\alpha$ is regular if and only if for all $A\in X/E$, there exists $B\in X/E$ such that $A\cap X\alpha\subseteq B\alpha$.
\end{theorem}

In general, the collection $\reg(T_E(X))$ of all regular elements in $T_E(X)$ fails to form a subsemigroup, because the composition of two regular transformations need not be regular. The following example illustrates this failure.
\begin{example}\label{ex: not closed}
	Let $X=\{1,2,3,4\}$, let $X/E=\{\{1,2\},\{3,4\}\}$, and define
	\[
		\alpha = \left(
		\begin{array}{ccccc}
			1 & 2 & 3 & 4 \\
			1 & 2 & 3 & 3 \\
		\end{array}
		\right),\
		\beta = \left(
		\begin{array}{ccccc}
			1 & 2 & 3 & 4 \\
			1 & 1 & 2 & 1 \\
		\end{array}
		\right).
	\]
	We can check by Theorem~\ref{thm: reg} that $\alpha$ and $\beta$ are regular in $T_E(X)$, but
	\[
		\alpha\beta = \left(
		\begin{array}{ccccc}
			1 & 2 & 3 & 4 \\
			1 & 1 & 2 & 2 \\
		\end{array}
		\right)
	\]
	is not regular. Hence the set of all regular elements in $T_E(X)$ is not closed under composition.
\end{example}

To construct our first maximal regular subsemigroup containing $G_E(X)$, we begin by recalling the double direction equivalence-preserving transformation semigroup $T_{E^*}(X)$ introduced in Section~\ref{sec:prelim}. By Deng, Zeng, and Xu \cite{Deng2010}*{Theorem 3.2}, the semigroup $T_{E^*}(X)$ is regular if and only if $|X/E|$ is finite. The regular part of $T_{E^*}(X)$, given by
\[
	\reg(T_{E^*}(X)) = \{\alpha \in T_{E^*}(X) : A \cap X\alpha \neq \emptyset \text{ for all } A \in X/E\},
\]
is a regular subsemigroup and hence the largest regular subsemigroup of $T_{E^*}(X)$ (see \cite{Chen2023}*{Lemma~2.5} and \cite{Sangkhanan2026}*{Theorems~2.1 and 2.2}). For our subsequent development in $T_E(X)$, it is particularly convenient to express $\reg(T_{E^*}(X))$ through the preimage structure of the $E$-classes, as established in the following lemma.

\begin{lemma}\label{lem: Reg TEstar regular}
	The subset
	\[
		\reg(T_{E^*}(X))=\{\alpha\in T_E(X) : \forall A\in X/E,\ A\alpha^{-1}\in X/E\}
	\]
	is a regular subsemigroup of $T_E(X)$.
\end{lemma}
\begin{proof}
	First, we verify the set equality
	\[
		\reg(T_{E^*}(X))=\{\alpha\in T_E(X):\forall A\in X/E,\ A\alpha^{-1}\in X/E\}.
	\]
	Let $\alpha \in T_E(X)$ be such that $A\alpha^{-1} \in X/E$ for all $A \in X/E$. To show $\alpha \in T_{E^*}(X)$, let $x,y \in X$. If $(x,y) \in E$, then $(x\alpha, y\alpha) \in E$ since $\alpha \in T_E(X)$. Conversely, if $(x\alpha, y\alpha) \in E$, then $x\alpha, y\alpha \in A$ for some $E$-class $A \in X/E$, which implies $x,y \in A\alpha^{-1}$. Since $A\alpha^{-1} \in X/E$, it follows that $(x,y) \in E$. Thus $\alpha \in T_{E^*}(X)$. Moreover, since $A\alpha^{-1} \neq \emptyset$ for each $A \in X/E$, we have $A \cap X\alpha \neq \emptyset$ for all $A \in X/E$. By \cite{Sangkhanan2026}*{Theorem 2.1}, $\alpha \in \reg(T_{E^*}(X))$.

	Conversely, let $\alpha \in \reg(T_{E^*}(X))$. Then $\alpha \in T_{E^*}(X) \subseteq T_E(X)$ and $A \cap X\alpha \neq \emptyset$ for all $A \in X/E$. Since $\alpha \in T_{E^*}(X)$, distinct $E$-classes are mapped into distinct $E$-classes, so $\alpha$ induces an injective mapping on the quotient set $X/E$. Furthermore, because $A \cap X\alpha \neq \emptyset$ for all $A \in X/E$, this induced mapping is surjective, and hence a bijection on $X/E$. Therefore, for each $A \in X/E$, there is a unique $B \in X/E$ such that $A\alpha^{-1} = B \in X/E$.

	This establishes the set equality. Finally, since $\reg(T_{E^*}(X))$ is a regular subsemigroup of $T_{E^*}(X)$ and $T_{E^*}(X) \subseteq T_E(X)$, it follows that $\reg(T_{E^*}(X))$ is a regular subsemigroup of $T_E(X)$.
\end{proof}

Although $\reg(T_{E^*}(X))$ is a regular subsemigroup of $T_E(X)$ containing $G_E(X)$, it is typically not maximal in $T_E(X)$. To enlarge $\reg(T_{E^*}(X))$ to a maximal regular subsemigroup, we introduce an auxiliary subset $L \subseteq T_E(X)$ consisting of transformations whose image meets at least one equivalence class in a single point, while any class meeting the image in more than one point has an entire equivalence class as its preimage. The following lemma confirms that $L$ is itself a regular subsemigroup of $T_E(X)$.

\begin{lemma}\label{lem: L regular subsemigroup}
	Let $L$ be the set of all transformations $\alpha \in T_E(X)$ satisfying the following properties:
	\begin{enumerate}[label=(\arabic*)]
		\item there exists $A\in X/E$ such that $|A\cap X\alpha|=1$;
		\item if $A\in X/E$ and $|A\cap X\alpha|>1$, then $A\alpha^{-1}\in X/E$.
	\end{enumerate}
	Then $L$ is a regular subsemigroup of $T_E(X)$.
\end{lemma}
\begin{proof}
	We first show that $L$ is closed under composition. Let $\alpha,\beta\in L$. Then there exists $A\in X/E$ such that $|A\cap X\alpha|=1$, which implies that $|(A\cap X\alpha)\beta|=1$. Let $A\cap X\alpha=\{x\alpha\}$ for some $x\in X$. Then $(A\cap X\alpha)\beta=\{x\alpha\beta\}$. There exists $B\in X/E$ such that $(A\cap X\alpha)\beta\subseteq B\cap X\beta$. We see that $x\alpha\beta\in B\cap X\alpha\beta\neq\emptyset$. Suppose for contradiction that $|B\cap X\alpha\beta|>1$. Then $|B\cap X\beta|\geq|B\cap X\alpha\beta|>1$. By condition~(2), we have $B\beta^{-1}\in X/E$, from which it follows that $B\beta^{-1}=A$. Since $|B\cap X\alpha\beta|>1$, there exists $y\alpha\beta\in B\cap X\alpha\beta$ such that $x\alpha\beta\neq y\alpha\beta$. Then $y\alpha\in B\beta^{-1}\cap X\alpha=A\cap X\alpha=\{x\alpha\}$. Thus $x\alpha\beta=y\alpha\beta$, a contradiction. Therefore, $|B\cap X\alpha\beta|=1$, and hence $\alpha\beta$ satisfies condition~(1).

	To verify condition~(2), let $A\in X/E$ be such that $|A\cap X\alpha\beta|>1$. Then $|A\cap X\beta|\geq|A\cap X\alpha\beta|>1$, which implies by condition~(2) that $A\beta^{-1}=B$ for some $E$-class $B$. Since $A\cap X\alpha\beta = (B\cap X\alpha)\beta$, we obtain $|B\cap X\alpha| \ge |A\cap X\alpha\beta| > 1$. By condition~(2) for $\alpha$, we have $B\alpha^{-1}=C$ for some $E$-class $C$. We have
	\[
		A(\alpha\beta)^{-1}=A\beta^{-1}\alpha^{-1}=B\alpha^{-1}=C\in X/E,
	\]
	and so $\alpha\beta$ satisfies condition~(2). Therefore, $\alpha\beta\in L$, as required.

	Next, we prove that $L$ is a regular subsemigroup of $T_E(X)$. Let $\alpha\in L$. By condition~(1), there exists an $E$-class $A_0\in X/E$ such that $|A_0\cap X\alpha|=1$. Let $A_0\cap X\alpha=\{a_0\}$, and choose $b_0\in X$ such that $b_0\alpha=a_0$. Let $B_0 = [b_0]_E \in X/E$. For each $A\alpha^{-1}\in E(\alpha)$, we consider the following two cases.

	\textbf{Case 1:} $|A\cap X\alpha|=1$. We can write
	\[
		\alpha|_{A\alpha^{-1}} = \left(
		\begin{array}{c}
			A\alpha^{-1} \\
			a            \\
		\end{array}
		\right)
	\]
	where $A\alpha^{-1}$ is a union of some $E$-classes in $X/E$ and $\{a\}=A\cap X\alpha$. We have $a=b\alpha$ for some $b\in X$ (if $A = A_0$, we choose $b = b_0$). Define $\beta|_A$ by
	\[
		\beta|_A = \left(
		\begin{array}{c}
			A \\
			b \\
		\end{array}
		\right).
	\]

	\textbf{Case 2:} $|A\cap X\alpha|>1$. Then $A\alpha^{-1}=B$ for some $E$-class $B$. Let $A\cap X\alpha=\{a_i : i\in I\}$ with $|I|>1$. For each $i\in I$, let $A_i = a_i\alpha^{-1} \cap B$, so that $\{A_i : i\in I\}$ forms a partition of $B$. We can write
	\[
		\alpha|_{A\alpha^{-1}} = \left(
		\begin{array}{c}
		A_i \\
		a_i \\
		\end{array}
		\right).
	\]
	For each $i\in I$, choose an element $b_i\in A_i$, so that $b_i\alpha = a_i$ with $b_i\in B$. Choose and fix $i_0\in I$, and define $\beta|_A$ by
	\[
		\beta|_A = \left(
		\begin{array}{cc}
		a_i & A\setminus\{a_i : i\in I\} \\
		b_i & b_{i_0}                     \\
		\end{array}
		\right).
	\]
	For each $E$-class $C$ such that $C\cap X\alpha=\emptyset$, define $\beta|_C$ by $x\beta=b_0$ for all $x\in C$.

	To verify that $\beta\in L$, we first show that $\beta\in T_E(X)$. For each $A\in X/E$, if $A$ satisfies Case~1, then $A\beta = \{b\} \subseteq [b]_E$; if $A$ satisfies Case~2, then $A\beta = \{b_i : i\in I\} \subseteq B$; and if $A\cap X\alpha=\emptyset$, then $A\beta = \{b_0\} \subseteq B_0$. In all cases, $A\beta$ is contained in a single $E$-class, so $\beta\in T_E(X)$.

	Next, we check condition~(1). Observe that $B_0\alpha \subseteq A_0$. If $A\in X/E$ satisfies Case~2 with $A\alpha^{-1} = B\in X/E$, then $B\alpha \subseteq A$ with $|A\cap X\alpha| > 1$, so $B \neq B_0$. If $A\in X/E$ satisfies Case~1 with $A\neq A_0$, then $a = b\alpha \in A$ implies $[b]_E \neq B_0$ (since otherwise $a = b\alpha \in B_0\alpha \subseteq A_0$, a contradiction). Thus, only $A_0$ and the classes $C$ with $C\cap X\alpha = \emptyset$ map into $B_0$, and each maps entirely to $b_0$. Hence $B_0\cap X\beta = \{b_0\}$, which gives $|B_0\cap X\beta| = 1$, verifying condition~(1).

	To check condition~(2), let $D\in X/E$ be such that $|D\cap X\beta| > 1$. Since $|B_0\cap X\beta| = 1$, we have $D \neq B_0$. Furthermore, for any class $A$ satisfying Case~1, we have $A\beta = \{b\}$ with $b\alpha \in A$; if $[b]_E = D$, then $D\alpha \subseteq A$, which implies that no other class from Case~1, no empty class, and no class from Case~2 can map into $D$, so $|D\cap X\beta| = 1$, a contradiction. Therefore, $D$ cannot be the $E$-class of any point chosen in Case~1, meaning that $D$ must contain elements from some class $A$ satisfying Case~2. Since $A\beta \subseteq A\alpha^{-1}$, we have $D = A\alpha^{-1} \in X/E$. Moreover, since $D\alpha \subseteq A$ with $|A\cap X\alpha| > 1$, the class $A$ is the unique $E$-class mapping into $D$. Hence $D\beta^{-1} = A\in X/E$, verifying condition~(2). Therefore, $\beta\in L$.

	Finally, we show that $\alpha = \alpha\beta\alpha$. Let $x\in X$. Then $x\in A\alpha^{-1}$ for some $A\alpha^{-1}\in E(\alpha)$. If $|A\cap X\alpha|=1$, then
	\[
		x\alpha\beta\alpha = a\beta\alpha = b\alpha = a=x\alpha.
	\]
	If $|A\cap X\alpha|>1$, then $A\alpha^{-1}=B$ for some $E$-class $B$. By Case~2, we have $x\in A_i$ for some $i\in I$ and hence
	\[
		x\alpha\beta\alpha = a_i\beta\alpha = b_i\alpha = a_i=x\alpha.
	\]
	Therefore, $\alpha = \alpha\beta\alpha$ since $x$ is arbitrary.
\end{proof}

To prove that $\reg(T_{E^*}(X)) \cup L$ is a maximal regular subsemigroup, we need the following lemma.

\begin{lemma}\label{lem: not regular by element in L}
	Let $\alpha\in T_E(X)$ and $A\in X/E$. If $|A\cap X\alpha|>1$ and $A\alpha^{-1}\notin X/E$, then there exists $\beta\in L$ such that $\beta\alpha$ is not regular.
\end{lemma}
\begin{proof}
	Assume that $|A\cap X\alpha|>1$ and $A\alpha^{-1}\notin X/E$. Since $A\alpha^{-1}$ is a union of at least two $E$-classes and $|A\cap X\alpha|>1$, the images of these $E$-classes under $\alpha$ cannot all coincide as the same singleton. Thus, there exist distinct $E$-classes $B,C\subseteq A\alpha^{-1}$ and distinct elements $x,y\in A\cap X\alpha$ such that $x\in B\alpha$ and $y\in C\alpha$. Let $x=b\alpha$ and $y=c\alpha$ for some $b\in B$ and $c\in C$. Define a transformation $\beta: X \to X$ by
	\[
		\beta = \left(
		\begin{array}{cc}
			B & X\setminus B \\
			b & c            \\
		\end{array}
		\right).
	\]
	We obtain $\beta\in L$ and $A\cap X\beta\alpha=\{x,y\}$. It is easy to verify that $A\cap X\beta\alpha\nsubseteq D\beta\alpha$ for all $D\in X/E$. Hence $\beta\alpha$ is not regular.
\end{proof}

Now, we are ready to prove our first main theorem.

\begin{theorem}\label{thm: Reg TEstar union L maximal}
	Let $E$ be an equivalence relation on $X$ that is neither the identity relation nor the universal relation on $X$. The set $\reg(T_{E^*}(X))\cup L$ is a maximal regular subsemigroup of $T_E(X)$.
\end{theorem}
\begin{proof}
	We prove that $\reg(T_{E^*}(X))\cup L$ is a subsemigroup of $T_E(X)$. Let $\alpha\in \reg(T_{E^*}(X))$ and $\beta\in L$. There exists $A\in X/E$ such that $|A\cap X\beta|=1$. We obtain $|A\cap X\alpha\beta|\leq|A\cap X\beta|=1$. Moreover, there is an $E$-class $B\subseteq A\beta^{-1}$ since $A\beta^{-1}$ is a union of $E$-classes. Hence $B\alpha^{-1}=C$ for some $E$-class $C$ since $\alpha\in \reg(T_{E^*}(X))$. We have
	\[
		\emptyset\neq C\alpha\beta\subseteq B\beta\cap X\alpha\beta\subseteq A\cap X\alpha\beta,
	\]
	and so $|A\cap X\alpha\beta|=1$. Thus $\alpha\beta$ satisfies condition~(1) of Lemma~\ref{lem: L regular subsemigroup}.

	To verify condition~(2) of Lemma~\ref{lem: L regular subsemigroup}, let $A\in X/E$ be such that $|A\cap X\alpha\beta|>1$. Then $|A\cap X\beta|\geq|A\cap X\alpha\beta|>1$, which implies that $A\beta^{-1}=B$ for some $E$-class $B$. Since $\alpha\in \reg(T_{E^*}(X))$, there exists $C\in X/E$ such that $B\alpha^{-1}=C$. Hence
	\[
		A(\alpha\beta)^{-1}=A\beta^{-1}\alpha^{-1}=B\alpha^{-1}=C\in X/E,
	\]
	from which it follows that $\alpha\beta$ satisfies condition~(2) of Lemma~\ref{lem: L regular subsemigroup}, as required. Hence $\alpha\beta\in L\subseteq \reg(T_{E^*}(X))\cup L$.

	Next, we show that $\beta\alpha \in L$. Since $\beta\in L$, there exists $A\in X/E$ such that $|A\cap X\beta|=1$. Let $A\cap X\beta=\{x\beta\}$ for some $x\in X$. We obtain $x\beta\alpha\in A\alpha\cap X\beta\alpha\subseteq B\cap X\beta\alpha$ for some $B\in X/E$. Suppose for contradiction that there exists $y\beta\alpha\in B\cap X\beta\alpha$ such that $x\beta\alpha\neq y\beta\alpha$. Then $|B\cap X\alpha|\geq|B\cap X\beta\alpha|>1$ and hence $B\alpha^{-1}=A$. We see that $x\beta,y\beta\in B\alpha^{-1}\cap X\beta=A\cap X\beta$. Thus $x\beta=y\beta$ since $|A\cap X\beta|=1$. So $x\beta\alpha=y\beta\alpha$, a contradiction. Therefore, $|B\cap X\beta\alpha|=1$, from which it follows that $\beta\alpha$ satisfies condition~(1) of Lemma~\ref{lem: L regular subsemigroup}.

	To verify condition~(2) of Lemma~\ref{lem: L regular subsemigroup}, let $A\in X/E$ be such that $|A\cap X\beta\alpha|>1$. Then there exist two distinct elements $x\beta\alpha$ and $y\beta\alpha$ in $A\cap X\beta\alpha$. Since $\alpha\in \reg(T_{E^*}(X))$, we have $A\alpha^{-1}=B$ for some $E$-class $B$. Hence $x\beta,y\beta\in A\alpha^{-1}\cap X\beta=B\cap X\beta$, which implies that $|B\cap X\beta|>1$. Then $B\beta^{-1}=C$ for some $C\in X/E$. We obtain
	\[
		A(\beta\alpha)^{-1}=A\alpha^{-1}\beta^{-1}=B\beta^{-1}=C\in X/E,
	\]
	and so $\beta\alpha$ satisfies condition~(2) of Lemma~\ref{lem: L regular subsemigroup}. Therefore, $\beta\alpha\in L$. Thus $\reg(T_{E^*}(X))\cup L$ is closed under composition. Since every element in $\reg(T_{E^*}(X))$ and $L$ is regular, it follows that $\reg(T_{E^*}(X))\cup L$ is a regular subsemigroup of $T_E(X)$. Moreover, since $E$ is neither the identity relation nor the universal relation on $X$, the semigroup $T_E(X)$ is not regular, so $\reg(T_{E^*}(X))\cup L \subseteq \reg(T_E(X)) \subsetneq T_E(X)$, which shows that $\reg(T_{E^*}(X))\cup L$ is a proper subsemigroup of $T_E(X)$.

	Finally, we prove the maximality. Let $S$ be a proper regular subsemigroup of $T_E(X)$ such that $\reg(T_{E^*}(X))\cup L\subseteq S \subsetneq T_E(X)$. Suppose for contradiction that $\reg(T_{E^*}(X))\cup L \neq S$, and pick $\alpha\in S\setminus(\reg(T_{E^*}(X))\cup L)$. Since $\alpha\notin L$, $\alpha$ fails condition~(1) or condition~(2) of Lemma~\ref{lem: L regular subsemigroup}. We consider the following two cases:

	\textbf{Case 1:} $\alpha$ fails condition~(2) of Lemma~\ref{lem: L regular subsemigroup}. Then there exists a class $A\in X/E$ such that $|A\cap X\alpha|>1$ and $A\alpha^{-1}\notin X/E$. By Lemma~\ref{lem: not regular by element in L}, there exists $\beta\in L\subseteq S$ such that $\beta\alpha$ is not regular in $T_E(X)$, contradicting the assumption that $S$ is regular.

	\textbf{Case 2:} $\alpha$ satisfies condition~(2) of Lemma~\ref{lem: L regular subsemigroup}, but fails condition~(1). Since $\alpha$ fails condition~(1), we have $|B\cap X\alpha|\neq 1$ for all $B\in X/E$. Since $\alpha\notin \reg(T_{E^*}(X))$, there exists an $E$-class $A_0\in X/E$ such that $A_0\alpha^{-1}\notin X/E$. Because $\alpha$ satisfies condition~(2), any $B\in X/E$ with $|B\cap X\alpha|>1$ must satisfy $B\alpha^{-1}\in X/E$. Thus, we must have $A_0\cap X\alpha = \emptyset$. Since $\alpha\in S$ and $S$ is regular, there exists $\beta\in S$ such that $\alpha=\alpha\beta\alpha$. Let
	\[
		\mathcal{C}=\{B\in X/E : B\cap X\alpha \neq \emptyset\}.
	\]
	Clearly, $\mathcal{C}\neq\emptyset$. For each $B\in\mathcal{C}$, we have $|B\cap X\alpha|>1$, so $B\alpha^{-1}\in X/E$. For any $x\alpha\in B\cap X\alpha$, $x\alpha\beta\alpha=x\alpha\in B$, so $x\alpha\beta\in B\alpha^{-1}\in X/E$. Thus $B\beta\subseteq B\alpha^{-1}$ for all $B\in\mathcal{C}$.

	Fix $a\in A_0$, and let $C = [a\beta]_E \in X/E$. Since $a\beta\alpha \in D\cap X\alpha$ for some $D\in\mathcal{C}$, we have $a\beta\in D\alpha^{-1}$. Since $D\in\mathcal{C}$, the preimage $D\alpha^{-1}$ is an $E$-class, which implies $D\alpha^{-1} = [a\beta]_E = C$. Then $a\beta\alpha\beta \in D\beta \subseteq D\alpha^{-1} = C$. Thus both $A_0\beta$ and $D\beta$ are contained in $C$, which implies $C\beta^{-1}\notin X/E$. Moreover, since $D\in\mathcal{C}$, we have $|D\cap X\alpha|>1$, so there exists an element $w\in D\cap X\alpha$ such that $w \neq a\beta\alpha$; write $w = b\alpha$ for some $b\in X$. If $a\beta=b\alpha\beta$, then $a\beta\alpha=b\alpha\beta\alpha=b\alpha=w$, a contradiction. Hence $a\beta\neq b\alpha\beta$. Furthermore, $b\alpha\beta \in D\beta\subseteq C$, so $a\beta, b\alpha\beta \in C\cap X\beta$, which gives $|C\cap X\beta|>1$. By Lemma~\ref{lem: not regular by element in L}, there exists $\gamma\in L\subseteq S$ such that $\gamma\beta$ is not regular in $T_E(X)$, which again contradicts the assumption that $S$ is regular.

	Therefore, $\reg(T_{E^*}(X))\cup L=S$, showing that $\reg(T_{E^*}(X))\cup L$ is a maximal regular subsemigroup of $T_E(X)$.
\end{proof}

\begin{remark}\label{rem: Reg TEstar union L contains units}
	We note that $G_E(X) \subseteq \reg(T_{E^*}(X)) \subseteq \reg(T_{E^*}(X))\cup L$, so the maximal regular subsemigroup $\reg(T_{E^*}(X))\cup L$ contains the entire group of units $G_E(X)$.
\end{remark}

Next, we construct another maximal regular subsemigroup containing $G_E(X)$, beginning with the following lemma for the set $M$.

\begin{lemma}\label{lem: reg M regular subsemigroup}
	The set
	\[
		M = \{\alpha\in \reg(T_E(X)) : X\alpha\subseteq A \text{ for some } A\in X/E\}
	\]
	is a regular subsemigroup of $T_E(X)$.
\end{lemma}
\begin{proof}
	First, we prove that $M$ is a subsemigroup of $T_E(X)$. Let $\alpha,\beta\in M$. Then $\alpha,\beta\in\reg(T_E(X))$, and $X\alpha\subseteq B$ for some $E$-class $B\in X/E$. Since $\beta\in T_E(X)$, we have $X\alpha\beta\subseteq B\beta\subseteq A_0$ for some $E$-class $A_0\in X/E$. To show $\alpha\beta\in\reg(T_E(X))$, we use Theorem~\ref{thm: reg}. Let $A\in X/E$. If $A\cap X\alpha\beta = \emptyset$, then $A\cap X\alpha\beta \subseteq C\alpha\beta$ holds trivially for any $C\in X/E$. If $A\cap X\alpha\beta \neq \emptyset$, then $A = A_0$, and $A_0\cap X\alpha\beta = X\alpha\beta = (X\alpha)\beta = (B\cap X\alpha)\beta$. Since $\alpha\in\reg(T_E(X))$, Theorem~\ref{thm: reg} implies that there exists an $E$-class $C\in X/E$ such that $B\cap X\alpha \subseteq C\alpha$. Thus
	\[
		A_0\cap X\alpha\beta = (B\cap X\alpha)\beta \subseteq (C\alpha)\beta = C\alpha\beta.
	\]
	By Theorem~\ref{thm: reg}, $\alpha\beta\in\reg(T_E(X))$. Hence $\alpha\beta\in M$, which proves that $M$ is a subsemigroup of $T_E(X)$.

	Next, we show that every element $\alpha\in M$ is regular in $M$. Let $\alpha\in M$. Then $X\alpha = \{a_i : i\in I\} \subseteq B$ for some $E$-class $B\in X/E$, where $\alpha$ is represented by
	\[
		\alpha = \left(
		\begin{array}{c}
		A_i\\
		a_i\\
		\end{array}
		\right).
	\]
	Since $\alpha\in\reg(T_E(X))$, Theorem~\ref{thm: reg} yields an $E$-class $C\in X/E$ such that $B\cap X\alpha = \{a_i : i\in I\} \subseteq C\alpha$. Thus for each $i\in I$, there exists $c_i\in C$ such that $a_i = c_i\alpha$. Fix an index $i_0\in I$, and let $Z = X\setminus\{a_i : i\in I\}$. Define a transformation $\beta: X \to X$ by
	\[
		\beta = \left(
		\begin{array}{cc}
			a_i & Z       \\
			c_i & c_{i_0} \\
		\end{array}
		\right).
	\]
	Note that $X\beta = \{c_i : i\in I\} \subseteq C \in X/E$, so $(x\beta, y\beta)\in E$ for all $x,y\in X$, which proves that $\beta\in T_E(X)$ and $X\beta\subseteq C$. Furthermore, $C\cap X\beta = \{c_i : i\in I\} \subseteq B\beta$ because $\{a_i : i\in I\} \subseteq B$ and $a_i\beta = c_i$ for all $i\in I$. By Theorem~\ref{thm: reg}, $\beta\in\reg(T_E(X))$, so $\beta\in M$. Finally, for any $x\in X$, $x\in A_i$ for some $i\in I$, so $x\alpha = a_i$, which gives
	\[
		x\alpha\beta\alpha = a_i\beta\alpha = c_i\alpha = a_i = x\alpha.
	\]
	Since $x\in X$ was arbitrary, $\alpha = \alpha\beta\alpha$. Therefore, $M$ is a regular subsemigroup of $T_E(X)$.
\end{proof}

Since the image of each element in $M$ is contained in a single $E$-class, $M$ does not contain $G_E(X)$. To include the group of units, we introduce the set $Q$ and show that it is a regular subsemigroup.

\begin{lemma}\label{lem: Q regular subsemigroup}
	Let
	\[
		\Omega=\{A\in X/E : |A|\geq 2\}
	\]
	and
	\[
		Q=\{\alpha\in T_E(X) : \alpha|_{\bigcup\Omega}\in\Sym(\textstyle\bigcup\Omega) \text{ and } A\alpha\in\Omega \text{ for all } A\in\Omega\}.
	\]
	Then $Q$ is a regular subsemigroup of $T_E(X)$.
\end{lemma}
\begin{proof}
	First, we show that $Q$ is a subsemigroup of $T_E(X)$. Let $\alpha,\beta\in Q$. Since $\alpha|_{\bigcup\Omega}$ and $\beta|_{\bigcup\Omega}$ belong to $\Sym(\bigcup\Omega)$, their composition $(\alpha\beta)|_{\bigcup\Omega} = (\alpha|_{\bigcup\Omega})(\beta|_{\bigcup\Omega})$ also belongs to $\Sym(\bigcup\Omega)$. Moreover, for any $A\in\Omega$, we have $A\alpha\in\Omega$, which implies $A(\alpha\beta) = (A\alpha)\beta\in\Omega$. Hence $\alpha\beta\in Q$, proving that $Q$ is closed under composition.

	Next, we show that $Q$ is regular. Let $\alpha\in Q$. To construct an element $\beta\in Q$ satisfying $\alpha=\alpha\beta\alpha$, we define $\beta|_A$ for each $E$-class $A\in X/E$ as follows:

	\textbf{Case 1:} $A\in\Omega$. Since $\alpha|_{\bigcup\Omega}\in\Sym(\bigcup\Omega)$ and $\alpha$ maps $E$-classes in $\Omega$ to $E$-classes in $\Omega$, there exists a unique class $B\in\Omega$ such that $B\alpha = A$, and the restriction $\alpha|_B: B\to A$ is a bijection. Let $B=\{b_i : i\in I\}$ and $a_i = b_i\alpha$ for each $i\in I$, so $A=\{a_i : i\in I\}$. We define $\beta|_A: A\to B$ by $a_i\beta = b_i$ for all $i\in I$. Notice that $A\beta = B\in\Omega$ and $\beta|_A = (\alpha|_B)^{-1}$.

	\textbf{Case 2:} $A\notin\Omega$ and $A\cap X\alpha\neq\emptyset$. Since $A\notin\Omega$, we have $A=\{a\}$ for some $a\in X$. Since $a\in X\alpha$, the preimage $A\alpha^{-1}$ is nonempty. Choose a fixed element $b\in A\alpha^{-1}$, and define $\beta|_A: A\to X$ by $a\beta = b$.

	\textbf{Case 3:} $A\notin\Omega$ and $A\cap X\alpha=\emptyset$. Since $A=\{a\}$, choose an arbitrary element $c\in X$ and define $\beta|_A: A\to X$ by $a\beta = c$.

	Combining $\beta|_A$ over all $A\in X/E$ defines a transformation $\beta: X\to X$. We first show that $\beta\in T_E(X)$. Let $x,y\in X$ with $(x,y)\in E$. If $x,y\in A\in\Omega$, then $x\beta,y\beta\in A\beta = B\in\Omega$, so $(x\beta,y\beta)\in E$. If $x,y\in A\notin\Omega$, then $|A|=1$, which implies $x=y$ and thus $x\beta=y\beta$, so $(x\beta,y\beta)\in E$. Hence $\beta\in T_E(X)$. Furthermore, for each $A\in\Omega$, Case 1 gives $A\beta = B\in\Omega$ and $\beta|_{\bigcup\Omega} = \bigcup_{A\in\Omega}(\alpha|_{A\alpha^{-1}})^{-1}\in\Sym(\bigcup\Omega)$, which shows that $\beta\in Q$.

	Finally, we verify that $\alpha=\alpha\beta\alpha$. Let $x\in X$, and set $a = x\alpha$. We consider two cases: If $x\in \bigcup\Omega$, then $x\in B$ for some $B\in\Omega$. Let $A = B\alpha$. Since $\alpha\in Q$, we have $A\in\Omega$. By Case 1, $B = \{b_i : i\in I\}$ and $A = \{a_i : i\in I\}$ where $a_i = b_i\alpha$. Since $x\in B$, we have $x = b_i$ for some $i\in I$, so $x\alpha = a_i$. By Case 1, $a_i\beta = b_i$, which gives $x\alpha\beta\alpha = a_i\beta\alpha = b_i\alpha = a_i = x\alpha$. If $x\notin \bigcup\Omega$, let $A = [a]_E \in X/E$. If $A\in\Omega$, then $a\in A$, so Case 1 gives $a\beta\alpha = a(\alpha|_B)^{-1}\alpha = a = x\alpha$. If $A\notin\Omega$, then $A=\{a\}$ with $A\cap X\alpha\neq\emptyset$, so $A$ satisfies Case 2, where $a\beta = b\in A\alpha^{-1}$ (so $b\alpha = a$), giving $x\alpha\beta\alpha = a\beta\alpha = b\alpha = a = x\alpha$. In all cases, $x\alpha\beta\alpha = x\alpha$. Consequently, $\alpha=\alpha\beta\alpha$, proving that $Q$ is a regular subsemigroup of $T_E(X)$.
\end{proof}

\begin{remark}\label{rem: Q characterization}
	We observe that $Q=\{\alpha\in\reg(T_E(X)) : \alpha|_{\bigcup\Omega}\in\Sym(\bigcup\Omega)\}$. Indeed, if $\alpha\in Q$, then $\alpha\in\reg(T_E(X))$ by Lemma~\ref{lem: Q regular subsemigroup}. Conversely, suppose $\alpha\in\reg(T_E(X))$ with $\alpha|_{\bigcup\Omega}\in\Sym(\bigcup\Omega)$. For each $A\in\Omega$, since $\alpha\in T_E(X)$, there exists $B\in X/E$ such that $A\alpha\subseteq B$. Since $\alpha|_{\bigcup\Omega}$ is injective and $|A|\geq 2$, we have $|B|\geq |A\alpha| = |A| \geq 2$, so $B\in\Omega$. Suppose for contradiction that $A\alpha \neq B$. Then there exists $b\in B\setminus A\alpha$. Since $b\in\bigcup\Omega$ and $\alpha|_{\bigcup\Omega}$ is surjective on $\bigcup\Omega$, we have $b = c\alpha$ for some $c\in\bigcup\Omega\setminus A$. Since $\alpha\in\reg(T_E(X))$, Theorem~\ref{thm: reg} yields $D\in X/E$ such that $B\cap X\alpha\subseteq D\alpha$. Since $A\alpha\subseteq B\cap X\alpha\subseteq D\alpha$ and $|A\alpha|\geq 2$, we must have $D\in\Omega$ and $D=A$, which implies $D\alpha = A\alpha$. But then $b = c\alpha \in B\cap X\alpha \subseteq D\alpha = A\alpha$, a contradiction. Hence $A\alpha = B\in\Omega$ for all $A\in\Omega$, proving that $\alpha\in Q$.
\end{remark}

Having established the regularity of both $M$ and $Q$, we now consider their union $M \cup Q$. The following lemma shows that this union remains closed under composition and thus forms a regular subsemigroup of $T_E(X)$.

\begin{lemma}\label{lem: reg M union Q}
	The set $M\cup Q$ is a regular subsemigroup of $T_E(X)$.
\end{lemma}
\begin{proof}
	Let $S = M\cup Q$. By Lemma~\ref{lem: reg M regular subsemigroup} and Lemma~\ref{lem: Q regular subsemigroup}, $M$ and $Q$ are regular subsemigroups of $T_E(X)$, so every element of $S$ is regular in $S$. Therefore, it suffices to prove that $S$ is closed under composition.

	Let $\alpha,\beta\in S$. If $\alpha,\beta\in M$ or $\alpha,\beta\in Q$, then $\alpha\beta\in S$ by Lemma~\ref{lem: reg M regular subsemigroup} and Lemma~\ref{lem: Q regular subsemigroup}, respectively. It remains to consider the two cross-cases:

	\textbf{Case 1:} $\alpha\in M$ and $\beta\in Q$. Since $\alpha\in M$, there exists an $E$-class $A\in X/E$ such that $X\alpha \subseteq A$, so $X\alpha = A\cap X\alpha$. Since $\beta \in T_E(X)$, we have $A\beta \subseteq B$ for some $E$-class $B\in X/E$, which gives $X\alpha\beta = (A\cap X\alpha)\beta \subseteq A\beta \subseteq B$. To show that $\alpha\beta\in\reg(T_E(X))$, let $U\in X/E$. If $U \neq B$, then $U\cap X\alpha\beta = \emptyset \subseteq V\alpha\beta$ holds trivially for any $V\in X/E$. If $U = B$, then $B\cap X\alpha\beta = X\alpha\beta = (A\cap X\alpha)\beta$. Since $\alpha\in\reg(T_E(X))$, Theorem~\ref{thm: reg} yields an $E$-class $C\in X/E$ such that $A\cap X\alpha \subseteq C\alpha$. Applying $\beta$, we obtain
	\[
		B\cap X\alpha\beta = (A\cap X\alpha)\beta \subseteq (C\alpha)\beta = C\alpha\beta.
	\]
	By Theorem~\ref{thm: reg}, $\alpha\beta\in\reg(T_E(X))$, so $\alpha\beta\in M \subseteq S$.

	\textbf{Case 2:} $\alpha\in Q$ and $\beta\in M$. Since $\beta\in M$, there exists $A\in X/E$ such that $X\beta \subseteq A$. Then $X\alpha\beta \subseteq X\beta \subseteq A$. Thus for any $U \neq A$, $U\cap X\alpha\beta = \emptyset$. For $U = A$, we have $A\cap X\alpha\beta = X\alpha\beta \subseteq A\cap X\beta$. Since $\beta\in\reg(T_E(X))$, Theorem~\ref{thm: reg} yields an $E$-class $B\in X/E$ such that $A\cap X\beta \subseteq B\beta$. We consider two cases for $B$.
	If $|B|=1$, say $B=\{b\}$, then $B\beta = \{b\beta\}$, so $A\cap X\alpha\beta \subseteq A\cap X\beta \subseteq B\beta = \{b\beta\}$. If $A\cap X\alpha\beta = \emptyset$, then $A\cap X\alpha\beta \subseteq C\alpha\beta$ holds trivially for any $C\in X/E$. If $A\cap X\alpha\beta \neq \emptyset$, then $A\cap X\alpha\beta = \{b\beta\}$; choosing an element $x\in X$ such that $x\alpha\beta = b\beta$ and letting $C = [x]_E \in X/E$, we have $b\beta = x\alpha\beta \in C\alpha\beta$, which gives $A\cap X\alpha\beta = \{b\beta\} \subseteq C\alpha\beta$.
	If $|B|\geq 2$, then $B\in\Omega$. Since $\alpha\in Q$, $\alpha|_{\bigcup\Omega}$ maps $\Omega$ bijectively onto $\Omega$, so there exists a unique class $D\in\Omega$ such that $D\alpha = B$. Therefore,
	\[
		A\cap X\alpha\beta \subseteq A\cap X\beta \subseteq B\beta = (D\alpha)\beta = D\alpha\beta.
	\]
	In either case, Theorem~\ref{thm: reg} implies that $\alpha\beta\in\reg(T_E(X))$, and hence $\alpha\beta\in M \subseteq S$.

	In all cases, $\alpha\beta\in S$, proving that $S$ is a regular subsemigroup of $T_E(X)$.
\end{proof}

To prove the maximality of $M \cup Q$, we need the following lemma.

\begin{lemma}\label{lem: reg onto}
	Let $\alpha\in\reg(T_E(X))\setminus M$. If $\bigcup\Omega\setminus\im(\alpha|_{\bigcup\Omega})\neq\emptyset$, then there exists $\beta\in M$ such that $\alpha\beta\notin\reg(T_E(X))$.
\end{lemma}
\begin{proof}
	Suppose $\bigcup\Omega\setminus\im(\alpha|_{\bigcup\Omega})\neq\emptyset$, and choose $a\in\bigcup\Omega\setminus\im(\alpha|_{\bigcup\Omega})$. Let $A = [a]_E \in\Omega$, so $|A|\geq 2$. 

	We first show that $A\setminus X\alpha\neq\emptyset$. Suppose for contradiction that $A\subseteq X\alpha$. Since $\alpha\in\reg(T_E(X))$, Theorem~\ref{thm: reg} yields an $E$-class $B\in X/E$ such that $A = A\cap X\alpha \subseteq B\alpha$. Since $|A|\geq 2$, we have $|B|\geq |B\alpha|\geq |A|\geq 2$, so $B\in\Omega$ and hence $B\subseteq\bigcup\Omega$. As $a\in A\subseteq B\alpha$, there exists $b_0\in B\subseteq\bigcup\Omega$ such that $b_0\alpha = a$, which implies $a\in\im(\alpha|_{\bigcup\Omega})$, a contradiction. Thus, $A\setminus X\alpha\neq\emptyset$.

	Choose $z\in A\setminus X\alpha$, and pick $w\in A\setminus\{z\}$ (which exists since $|A|\geq 2$). Since $\alpha\notin M$, the image $X\alpha$ is not contained in $A$, so there exists an $E$-class $C\in X/E$ such that $C\neq A$ and $C\cap X\alpha\neq\emptyset$. Define a transformation $\beta: X\to X$ by
	\[
		\beta = \left(
		\begin{array}{cc}
			C\cup\{z\} & X\setminus(C\cup\{z\}) \\
			z          & w                      \\
		\end{array}
		\right).
	\]
	Since $X\beta = \{z,w\}\subseteq A$, we have $(x\beta,y\beta)\in E$ for all $x,y\in X$, so $\beta\in T_E(X)$. Moreover, since $z\in C\cup\{z\}$ and $w\in A\setminus\{z\}\subseteq X\setminus(C\cup\{z\})$, we obtain $A\beta = \{z,w\} = X\beta$. Thus $A\cap X\beta \subseteq A\beta$, so Theorem~\ref{thm: reg} implies $\beta\in\reg(T_E(X))$, and hence $\beta\in M$.

	To show that $\alpha\beta\notin\reg(T_E(X))$, note that since $C\cap X\alpha\neq\emptyset$, there exists $u\in C\cap X\alpha$, so $u\beta = z$. Furthermore, since $\alpha\notin M$, $X\alpha\not\subseteq C$, so there exists $v\in X\alpha\setminus C$. Since $z\notin X\alpha$, we have $v\neq z$, which implies $v\in X\setminus(C\cup\{z\})$ and thus $v\beta = w$. Therefore, $X\alpha\beta = \{z,w\}\subseteq A$, and so $A\cap X\alpha\beta = \{z,w\}$. 

	Now let $D\in X/E$. Since $\alpha\in T_E(X)$, we have $D\alpha\subseteq K$ for some $E$-class $K\in X/E$. If $K = C$, then $D\alpha\subseteq C \subseteq C\cup\{z\}$, so $D\alpha\beta = \{z\}$. If $K = A$, then since $z\notin X\alpha$, we have $D\alpha\subseteq A\setminus\{z\}\subseteq X\setminus(C\cup\{z\})$, so $D\alpha\beta = \{w\}$. If $K\notin\{A,C\}$, then $K\cap(C\cup\{z\}) = \emptyset$, so $D\alpha\subseteq X\setminus(C\cup\{z\})$, which gives $D\alpha\beta = \{w\}$. In all cases, $|D\alpha\beta| = 1$. If $\alpha\beta\in\reg(T_E(X))$, then Theorem~\ref{thm: reg} would yield an $E$-class $D\in X/E$ such that $\{z,w\} = A\cap X\alpha\beta \subseteq D\alpha\beta$, which contradicts $|D\alpha\beta| = 1$. Therefore, $\alpha\beta\notin\reg(T_E(X))$.
\end{proof}

Now, we are ready to prove our second main theorem.

\begin{theorem}\label{thm: M union Q maximal}
	Let $E$ be an equivalence relation on $X$ that is neither the identity relation nor the universal relation on $X$. The set $M\cup Q$ is a maximal regular subsemigroup of $T_E(X)$.
\end{theorem}
\begin{proof}
	By Lemma~\ref{lem: reg M union Q}, $M\cup Q$ is a regular subsemigroup of $T_E(X)$. Moreover, since $E$ is neither the identity relation nor the universal relation on $X$, the semigroup $T_E(X)$ is not regular, so $M\cup Q \subseteq \reg(T_E(X)) \subsetneq T_E(X)$, which shows that $M\cup Q$ is a proper subsemigroup of $T_E(X)$.

	Let $S$ be a proper regular subsemigroup of $T_E(X)$ such that $M\cup Q\subseteq S \subsetneq T_E(X)$. To show that $S=M\cup Q$, it suffices to prove that $\alpha\in M\cup Q$ for all $\alpha\in S$. Let $\alpha\in S$. Since $S$ is regular, we have $\alpha\in\reg(T_E(X))$. Suppose $\alpha\notin M$, so $\alpha\in\reg(T_E(X))\setminus M$. If $\bigcup\Omega\setminus\im(\alpha|_{\bigcup\Omega})$ is nonempty, then Lemma~\ref{lem: reg onto} yields an element $\beta\in M\subseteq S$ such that $\alpha\beta\notin\reg(T_E(X))$, contradicting $S\subseteq\reg(T_E(X))$. Therefore, $\bigcup\Omega\subseteq\im(\alpha|_{\bigcup\Omega})$.

	Next, we show that $\alpha|_{\bigcup\Omega}$ is injective. Let $s,t\in\bigcup\Omega$ be such that $s\alpha=t\alpha$. Since $\alpha\in S$, there exists $\gamma\in S$ such that $\alpha=\alpha\gamma\alpha$. Note that $\gamma\notin M$ (otherwise $X\gamma \subseteq B$ for some $B\in X/E$, which implies $X\alpha = (X\alpha\gamma)\alpha \subseteq (X\gamma)\alpha \subseteq B\alpha \subseteq A$ for some $A\in X/E$, giving $\alpha\in M$), so $\gamma\in\reg(T_E(X))\setminus M$. If $\bigcup\Omega\setminus\im(\gamma|_{\bigcup\Omega})\neq\emptyset$, Lemma~\ref{lem: reg onto} would yield $\lambda\in M\subseteq S$ such that $\gamma\lambda\notin\reg(T_E(X))$, contradicting $\gamma\lambda\in S\subseteq\reg(T_E(X))$. Hence $\bigcup\Omega\subseteq\im(\gamma|_{\bigcup\Omega})$, which implies $s,t\in \im(\gamma|_{\bigcup\Omega})$, so $s=x\gamma$ and $t=y\gamma$ for some $x,y\in\bigcup\Omega$. Since $\bigcup\Omega\subseteq\im(\alpha|_{\bigcup\Omega})$, we have $x=a\alpha$ and $y=b\alpha$ for some $a,b\in\bigcup\Omega$. Then
	\begin{align*}
		s & = x\gamma = a\alpha\gamma = a\alpha\gamma\alpha\gamma = x\gamma\alpha\gamma = s\alpha\gamma, \\
		t & = y\gamma = b\alpha\gamma = b\alpha\gamma\alpha\gamma = y\gamma\alpha\gamma = t\alpha\gamma.
	\end{align*}
	Since $s\alpha=t\alpha$, we get $s=s\alpha\gamma=t\alpha\gamma=t$, proving that $\alpha|_{\bigcup\Omega}$ is injective.

	Finally, we show that $\bigcup\Omega=\im(\alpha|_{\bigcup\Omega})$. Suppose for contradiction that there exists an element $a\in\im(\alpha|_{\bigcup\Omega})\setminus\bigcup\Omega$. Write $a=x\alpha$ for some $x\in\bigcup\Omega$. Let $A = [a]_E$ and $B = [x]_E$. Since $a\notin\bigcup\Omega$ and $x\in\bigcup\Omega$, we have $A=\{a\}$ and $|B|\geq 2$, so $B\in\Omega$. Pick $y\in B$ with $y\neq x$. Since $x,y\in B$, we have $(x,y)\in E$. As $\alpha\in T_E(X)$, $(x\alpha,y\alpha)\in E$. But $x\alpha=a\in A=\{a\}$, so $y\alpha=a=x\alpha$. Since $x,y\in B\subseteq\bigcup\Omega$, the injectivity of $\alpha|_{\bigcup\Omega}$ implies $x=y$, a contradiction. Thus $\bigcup\Omega=\im(\alpha|_{\bigcup\Omega})$, which proves $\alpha|_{\bigcup\Omega}\in\Sym(\bigcup\Omega)$.

	Since $\alpha\in\reg(T_E(X))$ and $\alpha|_{\bigcup\Omega}\in\Sym(\bigcup\Omega)$, Remark~\ref{rem: Q characterization} implies that $\alpha\in Q$. Therefore, $\alpha\in M\cup Q$, completing the proof.
\end{proof}

\begin{remark}\label{rem: M union Q contains units}
	We observe that $M\cup Q$ contains the group of units $G_E(X)$. Indeed, for any $\alpha \in G_E(X)$, Theorem~\ref{thm: group of units} implies that $\alpha\in\Sym(X)$ and $\alpha$ maps each $E$-class $A\in X/E$ bijectively onto an $E$-class $B\in X/E$ with $|B|=|A|$. In particular, $\alpha$ maps every $E$-class in $\Omega=\{A\in X/E : |A|\geq 2\}$ bijectively onto an $E$-class in $\Omega$, so $A\alpha\in\Omega$ for all $A\in\Omega$ and $\alpha|_{\bigcup\Omega}\in\Sym(\bigcup\Omega)$. Thus $\alpha\in Q$, showing that
	\[
		G_E(X) \subseteq Q \subseteq M\cup Q.
	\]
	Therefore, $M\cup Q$ contains the entire group of units $G_E(X)$.
\end{remark}

For convenience in the subsequent discussion, we denote the two maximal regular subsemigroups constructed above by
\[
	\mathcal{M}_1 = \reg(T_{E^*}(X)) \cup L \quad \text{and} \quad \mathcal{M}_2 = M \cup Q.
\]
By Theorems~\ref{thm: Reg TEstar union L maximal} and~\ref{thm: M union Q maximal} together with Remarks~\ref{rem: Reg TEstar union L contains units} and~\ref{rem: M union Q contains units}, both $\mathcal{M}_1$ and $\mathcal{M}_2$ are maximal regular subsemigroups of $T_E(X)$ containing the group of units $G_E(X)$.

Note that the maximality of $\mathcal{M}_1$ and $\mathcal{M}_2$ holds for an arbitrary set $X$ and an arbitrary equivalence relation $E$ that is neither the identity relation nor the universal relation on $X$.

A natural question arises: \emph{Are there any other maximal regular subsemigroups of $T_E(X)$ containing $G_E(X)$?} The following theorem shows that when $X$ is finite and $E$ is uniform with two classes, $\mathcal{M}_1$ and $\mathcal{M}_2$ are the only such subsemigroups.

\begin{theorem}\label{thm: classification two blocks}
	Let $X$ be a finite set and $E$ a uniform equivalence relation on $X$ with $|X/E| = 2$ such that each equivalence class has cardinality $k \geq 2$. A subsemigroup $S$ of $T_E(X)$ is a maximal regular subsemigroup containing $G_E(X)$ if and only if
	\[
		S = \mathcal{M}_1 \quad \text{or} \quad S = \mathcal{M}_2.
	\]
\end{theorem}
\begin{proof}
	$(\Leftarrow)$ By Theorems~\ref{thm: Reg TEstar union L maximal} and~\ref{thm: M union Q maximal}, both $\mathcal{M}_1$ and $\mathcal{M}_2$ are maximal regular subsemigroups of $T_E(X)$. By Remarks~\ref{rem: Reg TEstar union L contains units} and~\ref{rem: M union Q contains units}, both $\mathcal{M}_1$ and $\mathcal{M}_2$ contain $G_E(X)$. Hence each of $\mathcal{M}_1$ and $\mathcal{M}_2$ is a maximal regular subsemigroup of $T_E(X)$ containing $G_E(X)$.

	$(\Rightarrow)$ Let $S$ be a maximal regular subsemigroup of $T_E(X)$ such that $G_E(X) \subseteq S$. Since $|A_1| = |A_2| \geq 2$, every equivalence class belongs to $\Omega$, so $\Omega = X/E$ and $\bigcup\Omega = X$. It follows from the definition of $Q$ and Theorem~\ref{thm: group of units} that $Q = G_E(X)$, and thus $\mathcal{M}_2 = M \cup G_E(X)$. We first show that
	\[
		\reg(T_E(X)) = \mathcal{M}_1 \cup \mathcal{M}_2.
	\]
	Since $\mathcal{M}_1 \subseteq \reg(T_E(X))$ and $\mathcal{M}_2 \subseteq \reg(T_E(X))$, the inclusion $\mathcal{M}_1 \cup \mathcal{M}_2 \subseteq \reg(T_E(X))$ is clear.
	Conversely, let $\alpha \in \reg(T_E(X))$, and let $X/E = \{A_1, A_2\}$.
	If the image of $\alpha$ is contained in a single $E$-class, then $X\alpha \subseteq A$ for some $A \in X/E$, so $\alpha \in M \subseteq \mathcal{M}_2$ by definition of $M$.
	Otherwise, the image $X\alpha$ meets both classes $A_1$ and $A_2$.
	Since each domain equivalence class maps into a single $E$-class, one equivalence class must map into $A_1$ and the other into $A_2$.
	Thus $\alpha$ induces a bijection on $X/E$, which means $\alpha \in T_{E^*}(X)$.
	Since $|X/E| = 2$ is finite, by Deng, Zeng, and Xu \cite{Deng2010}*{Theorem 3.2}, $T_{E^*}(X)$ is regular, so $\alpha \in \reg(T_{E^*}(X)) \subseteq \mathcal{M}_1$.
	Therefore, every element of $\reg(T_E(X))$ belongs to $\mathcal{M}_1 \cup \mathcal{M}_2$, which establishes that $\reg(T_E(X)) = \mathcal{M}_1 \cup \mathcal{M}_2$.
	Since $S \subseteq \reg(T_E(X))$, we obtain
	\[
		S \subseteq \mathcal{M}_1 \cup \mathcal{M}_2.
	\]
	We claim that
	\[
		S \subseteq \mathcal{M}_1 \quad \text{or} \quad S \subseteq \mathcal{M}_2.
	\]
	Suppose for contradiction that $S \not\subseteq \mathcal{M}_1$ and $S \not\subseteq \mathcal{M}_2$.
	Then there exist an element $\alpha_1 \in S \cap (\mathcal{M}_1 \setminus \mathcal{M}_2)$ and an element $\alpha_2 \in S \cap (\mathcal{M}_2 \setminus \mathcal{M}_1)$.

	We first analyze $\alpha_2 \in \mathcal{M}_2 \setminus \mathcal{M}_1$.
	Since $\mathcal{M}_2 = M \cup G_E(X)$ and $G_E(X) \subseteq \mathcal{M}_1$, we have $\alpha_2 \in M \setminus \mathcal{M}_1$.
	By definition of $M$, $\alpha_2 \in \reg(T_E(X))$ and $X\alpha_2 \subseteq A_0$ for some $A_0 \in \{A_1, A_2\}$.
	If $|X\alpha_2| = 1$, then $|A_0 \cap X\alpha_2| = 1$, so condition~(1) of Lemma~\ref{lem: L regular subsemigroup} holds, and condition~(2) holds trivially, which implies $\alpha_2 \in L \subseteq \mathcal{M}_1$, a contradiction.
	Thus $|X\alpha_2| \geq 2$.
	Since $X = A_1 \cup A_2$, there exist $u \in A_1$ and $v \in A_2$ such that $u\alpha_2 \neq v\alpha_2$.
	Set $x = u\alpha_2$ and $y = v\alpha_2$, so that $x, y \in A_0$ with $x \neq y$.

	We next analyze $\alpha_1 \in \mathcal{M}_1 \setminus \mathcal{M}_2$.
	Since $\alpha_1 \notin \mathcal{M}_2 = M \cup G_E(X)$, the image of $\alpha_1$ is not contained in a single $E$-class, so $X\alpha_1$ meets both $A_1$ and $A_2$, and $\alpha_1 \notin G_E(X)$.
	Because $X/E = \{A_1, A_2\}$, the condition that $X\alpha_1$ meets both classes implies that $\alpha_1$ maps $A_1$ into one class and $A_2$ into the other, so $\alpha_1 \in \reg(T_{E^*}(X))$ induces a permutation on $\{A_1, A_2\}$.
	Since $\alpha_1 \notin G_E(X)$ and $|A_1| = |A_2| = k$, $\alpha_1$ cannot be injective on every equivalence class.
	Without loss of generality, we may assume that $\alpha_1$ is not injective on $A_1$, so there exist distinct elements $a, b \in A_1$ such that $a\alpha_1 = b\alpha_1$.
	Choose a permutation $\rho \in G_E(X)$ swapping $A_1$ and $A_2$ if $A_1\alpha_1 \subseteq A_1$ (or setting $\rho = \id_X$ if $A_1\alpha_1 \subseteq A_2$), so that the composite $\phi = \alpha_1\rho \in \langle G_E(X), \alpha_1 \rangle \subseteq S$ satisfies $A_1\phi \subseteq A_2$ and $A_2\phi \subseteq A_1$, while still satisfying $a\phi = b\phi$.
	Among all transformations in $\langle G_E(X), \alpha_1 \rangle$, choose an element $\psi$ that minimizes the image size
	\[
		|X\psi| = |A_1\psi| + |A_2\psi|.
	\]
	We claim that $|A_1\psi| = 1$ and $|A_2\psi| = 1$.
	Indeed, suppose for contradiction that $|A_i\psi| \geq 2$ for some $i \in \{1, 2\}$.
	Then there exist distinct elements $c, d \in A_i\psi$.
	Since $\psi \in T_E(X)$, the image $A_i\psi$ is contained in a single equivalence class, so $c$ and $d$ belong to the same $E$-class (either $A_1$ or $A_2$).
	Because $|A_1| = |A_2| \geq 2$, there exists a permutation $\pi \in G_E(X)$ such that $c\pi = a$ and $d\pi = b$ (specifically, $\pi$ maps each class into itself if $c, d \in A_1$, and $\pi$ swaps $A_1$ and $A_2$ if $c, d \in A_2$).
	Now consider the transformation $\psi' = \psi\pi\phi \in \langle G_E(X), \alpha_1 \rangle$.
	Since $a\phi = b\phi$, we have
	\[
		c(\pi\phi) = a\phi = b\phi = d(\pi\phi).
	\]
	Thus $\pi\phi$ identifies the two distinct elements $c, d \in X\psi$, which implies that
	\[
		|X\psi'| < |X\psi|,
	\]
	contradicting the minimality of $|X\psi|$.
	Hence $|A_1\psi| = 1$ and $|A_2\psi| = 1$.
	Since $\psi$ induces a permutation on $\{A_1, A_2\}$, the singletons $A_1\psi = \{w_1\}$ and $A_2\psi = \{w_2\}$ lie in distinct equivalence classes.
	Replacing $\psi$ by $\psi^2$ if necessary (which induces the identity permutation on $X/E$), we may assume that $w_1 \in A_1$ and $w_2 \in A_2$.

	Now choose a permutation $\sigma \in G_E(X)$ such that $w_1\sigma = u \in A_1$ and $w_2\sigma = v \in A_2$.
	Consider the transformation
	\[
		\theta = \psi\sigma\alpha_2 \in \langle G_E(X), \alpha_1, \alpha_2 \rangle \subseteq S.
	\]
	Since $X\alpha_2 \subseteq A_0$, we have $X\theta \subseteq A_0$, so the image of $\theta$ is contained in a single $E$-class.
	Since $A_1\psi = \{w_1\}$ and $A_2\psi = \{w_2\}$, for any $t_1 \in A_1$ and $t_2 \in A_2$, we have
	\begin{align*}
		t_1\theta & = (t_1\psi)\sigma\alpha_2 = (w_1\sigma)\alpha_2 = u\alpha_2 = x, \\
		t_2\theta & = (t_2\psi)\sigma\alpha_2 = (w_2\sigma)\alpha_2 = v\alpha_2 = y.
	\end{align*}
	This implies that $A_1\theta = \{x\}$ and $A_2\theta = \{y\}$, so that
	\[
		X\theta = A_1\theta \cup A_2\theta = \{x, y\}.
	\]
	Since $x \neq y$, we have $|X\theta| = 2$, while $|A_1\theta| = 1$ and $|A_2\theta| = 1$.
	Consequently, no domain equivalence class $A_i \in X/E$ can satisfy $X\theta \subseteq A_i\theta$, because $|A_i\theta| = 1 < 2 = |X\theta|$.
	By Theorem~\ref{thm: reg}, $\theta \notin \reg(T_E(X))$.
	Since $\theta \in S$, this contradicts $S \subseteq \reg(T_E(X))$.
	This contradiction establishes the claim, so that
	\[
		S \subseteq \mathcal{M}_1 \quad \text{or} \quad S \subseteq \mathcal{M}_2.
	\]
	Since $\mathcal{M}_1$ and $\mathcal{M}_2$ are regular subsemigroups of $T_E(X)$ containing $G_E(X)$, the maximality of $S$ as a regular subsemigroup containing $G_E(X)$ implies that
	\[
		S = \mathcal{M}_1 \quad \text{or} \quad S = \mathcal{M}_2.
	\]
	This completes the proof.
\end{proof}

\subsection*{Open Problems}
In this paper, we have completely determined all maximal regular subsemigroups of $T_E(X)$ containing the group of units $G_E(X)$ for the case where $X$ is finite and $E$ is a uniform equivalence relation with two classes of cardinality at least $2$. In light of these results, several natural questions and directions for future research arise:
\begin{enumerate}[label=(\arabic*)]
	\item \emph{Uniform equivalence relations with more than two classes.}
	      Let $X$ be a finite set and $E$ a uniform equivalence relation on $X$ with $|X/E| \geq 3$ such that each equivalence class has cardinality at least $2$. Are $\mathcal{M}_1$ and $\mathcal{M}_2$ the only maximal regular subsemigroups of $T_E(X)$ containing $G_E(X)$, or do other maximal regular subsemigroups appear?

	\item \emph{Non-uniform equivalence relations.}
	      When $X$ is finite and $X/E = \{A_1, A_2\}$ with $|A_1| \neq |A_2|$ and $\min(|A_1|, |A_2|) \geq 2$, the group of units $G_E(X)$ cannot interchange the two classes. What is the complete classification of all maximal regular subsemigroups of $T_E(X)$ containing $G_E(X)$ in this setting?

	\item \emph{Infinite sets and infinite families of classes.}
	      Characterize the maximal regular subsemigroups of $T_E(X)$ containing $G_E(X)$ when the underlying set $X$ is infinite (with either finitely or infinitely many equivalence classes).
\end{enumerate}

\subsection*{Acknowledgments} This research was supported by Chiang Mai University.

\subsection*{Conflict of interest} The author declares no conflict of interest.

\bibliography{References}

\end{document}